\documentclass[11pt,a4paper,nocite]{article}

\usepackage{enumerate}
\usepackage{color}

\usepackage{amsmath}
\usepackage{amsthm}
\usepackage{amssymb}
\usepackage{amsfonts}

\newcommand{\RR}{{\mathbb R}}

\newcommand{\ZZ}{{\mathbb Z}}
\newcommand{\NN}{{\mathbb N}}

\newcommand{\BB}{{\mathcal B}}
\newcommand{\SA}{{\mathcal S}}

\newcommand{\D}{{\Delta}}

\newcommand{\ve}{{\varepsilon}}

\newcommand{\intt}{{\rm int\,}}

\newcommand{\BZ}{{\mathcal B}_0}

\newcommand{\OD}{\overline{D}}

\newtheorem{theorem}{Theorem}
\newtheorem{corollary}{Corollary}
\newtheorem{lemma}{Lemma}
\newtheorem{proposition}{Proposition}
\newtheorem{definition}{Definition}

\newtheorem{problem}{Problem}
\theoremstyle{definition}

\newtheorem{remark}{Remark}

\begin{document}

\title{Kahane's upper density and syndetic sets in LCA groups}

\author{Szil\' ard Gy. R\' ev\' esz\thanks{Supported in part by
the Hungarian National Office for Research, Development and Innovation, Project no. \# K-132097}}

\date{}

\maketitle


\begin{abstract}
Asymptotic uniform upper density, shortened as a.u.u.d., or simply upper density, is a classical notion which was first introduced by Kahane for sequences in the real line.

Syndetic sets were defined by Gottschalk and Hendlund.
For a locally compact group $G$, a set $S\subset G$ is syndetic, if there exists a compact subset $C\Subset G$ such that $SC=G$. Syndetic sets play an important role in various fields of applications of topological groups and semigroups, ergodic theory and number theory. A lemma in the book of Fürstenberg says that once a subset $A \subset \ZZ$ has positive a.u.u.d., then its difference set $A-A$ is syndetic.

The construction of a reasonable notion of a.u.u.d. in general locally compact Abelian groups (LCA groups for short) was not known for long, but in the late 2000's several constructions were worked out to generalize it from the base cases of $\ZZ^d$ and $\RR^d$. With the notion available, several classical results of the Euclidean setting became accessible even in general LCA groups.

Here we work out various versions in a general locally compact Abelian group $G$
of the classical statement that if a set $S\subset G$ has positive asymptotic uniform upper density, then the difference set $S-S$ is syndetic.
\end{abstract}

{\bf MSC 2020 Subject Classification.} Primary 22B05; Secondary
22B99, 05B10.

{\bf Keywords and phrases.} {\it density, asymptotic uniform upper
density, locally compact Abelian group, difference set, syndetic set.}

\bigskip

\section{Introduction}

The notion of syndetic sets was introduced in the fundamental book of Gottschalk and Hedlund \cite{GH}. A subset $S\subset G$ in a topological Abelian
(semi)group is a {\em syndetic} set, if there exists a compact set
$K\subset G$ such that for each element $g\in G$ there exists a
$k\in K$ with $gk\in S$; in other words, in topological groups
$\cup_{k\in K} Sk^{-1}=G$.

Fürstenberg presents as Proposition 3.19 (a) of \cite{Furst} the following.

\begin{proposition}\label{prop:Furst} Let $S\subset
\ZZ$ be a set having positive asymptotic uniform upper density $\overline{D}(S)>0$. Then $S-S$ is syndetic.
\end{proposition}

Here asymptotic uniform upper density\footnote{It is often, but seemingly erroneously called Banach density, among others also by Fürstenber.} stands for $\overline{D}(S):=\lim_{r\to \infty} \sup_{x \in \ZZ} \frac{[x-r,x+r] \cap S|}{2r}$. Generalization of the notion to $\ZZ^d$ and $\RR^d$, as well as for cases of non-discrete sets $S$ when the numerator becomes the Lebesgue measure instead of the cardinality measure, are straightforward. So is the extension of the proposition to these cases.

However, the notion of syndetic sets is even more general, and found many applications in several areas including dynamical systems, number theory, harmonic analysis. Still, an appropriate generalization of this proposition to general topological groups was not known, because there was no reasonably general and suitable notion of upper density.

In the following we discuss two generalized notions of asymptotic uniform upper density. on arbitrary LCA groups, which appeared only some fifteen years ago. With these generalized notions of Kahane's density, we present various generalized versions of the above result. In this paper we are going to prove generalizations of the above statement which typically read as follows.

\begin{theorem}\label{thm:intorgenth}
Let $G$ be a LCA group and $S\subset G$ a set with positive asymptotic uniform upper density:
$\overline{D}(S)>0$. Then the difference set $S-S$ is a syndetic set.
\end{theorem}

In this, everything seems to be quite clear even in the generality of LCA groups -- except for the right definition or construction of the upper density. So, the bulk of the paper will be devoted to sufficiently explain and describe various generalized notions of asymptotic uniform upper density in locally compact Abelian groups.

The structure of the paper is as follows. In Section \ref{sec:classical} we recall the classical notion of Kahane's density, in Section \ref{sec:appearence} we describe how the generalizations occurred, in Section \ref{sec:auudLCA} and \ref{sec:auud-discrete} we explain the generalized density notions. Then in Section \ref{sec:additiveresults} we recall results on difference sets, known to be valid in $\ZZ$ or in some other special cases and including the above Proposition from Fürstenberg's book. These we extend to LCA groups with the new generalized density notions in Sections \ref{sec:additiveproposgeneralized} and \ref{sec:Furst-extension-auud}. On our way we also prove some properties of our density notions, like e.g. subadditivity, which do not seem to be so obvious from their abstract, somewhat tricky definition.

\section{The classical notion of a.u.u.d.}\label{sec:classical}

The notion of asymptotic uniform upper density -- \emph{a.u.u.d. for short} -- of real sequences first appeared in the PhD thesis \cite{KahaneThese} of J.-P. Kahane in 1954, see also \cite{KahaneThAnnIF}. Other early, but definitely later appearances of the notion can be seen in e.g. \cite{Groemer}, \cite{Landau}, \cite{beurling}, \cite{beurlingB}, \cite{Furst}.

\medskip
Although his first construction was different, Kahane immediately shows \cite[Ch. I, \S 3, no. 1, p. 20]{KahaneThese} that the notion can be equivalently defined as follows\footnote{It seems that in the literature almost exclusively this latter equivalent form of the definition is used, although Kahane's original formulation is quite useful.}.
\begin{definition}[Kahane]\label{def:K2} If $S\subset \RR$ is a uniformly discrete sequence, then
\begin{equation}\label{RUNdensity}
\overline{D}^{\#}(S) :=
\limsup_{r\to\infty} \frac{\sup_{x\in\RR}  \# \{s\in S~:~ |s-x|\leq r\}}{2r}.
\end{equation}
\end{definition}
In fact, Kahane uses $\limsup_{|x|\to\infty}$ also in place of $\sup_{x\in\RR}$, but these variants are easily seen to be equivalent.
Also, the $\limsup$ is actually a limit, for the quantity essentially decreases in function of $r$. It is clear that a.u.u.d. is a translation invariant notion.

Kahane used the notion in harmonic analysis, and that remained a major field of applications ever since. Quite fast several related results appeared and the notion proved to be very fruitfully applied in seemingly different questions. A typical area of application is investigation of differences (additive bases, difference sets, packing and tiling, sets avoiding certain prescribed distances etc.), which can easily be understood if we e.g. note that in case $\overline{D}^{\#}(S)>0$, then $S-S$ has positive asymptotic upper density\footnote{This is defined in $\RR^d$ or $\ZZ^d$ as $\limsup_{r\to \infty} \frac{\# \{s\in S~:~ |s||\leq r\}}{2r}$, without taking a sup with respect to the center of the interval.} $\overline{d}^{\#}(S-S)\geq \overline{D}^{\#}(S)$, which is already a quite strong property. Still another area of application occurs in ergodic theory \cite{Furst}.

\medskip
As is already mentioned, in $\RR^d$ or $\ZZ^d$ one can analogously consider, with a fixed
basic set $K\subset \RR^d$ like e.g. the unit ball or unit cube,
\begin{equation}\label{RUNdensity2}
\overline{D}^{\#}_{K}(S) := \limsup_{r\to\infty}
\frac{\sup_{x\in\RR^d} {\rm \#} (S \cap (rK+x))}{|rK|}~.
\end{equation}
Here $K\subset\RR^d$ can be e.g. any fat body, or can even be more general, while for any Lebesgue measurable set $A\subset \RR^d$ $|A|$ stands for the Lebesgue measure of $A$. It is also well-known, that $\overline{D}^{\#}_{K}(S)$ \emph{gives the same value} for all nice sets $K\subset \RR^d$ (although this fact does not seem immediate from the formulation). To prove directly, it requires some tedious $\varepsilon$-covering of the boundary of $K_1$ by homothetic copies of $K_2$ etc. Landau works out this direct proof as \cite[Lemma 4]{Landau} in the generality of compact sets $K\Subset \RR^d$ normalized to have unit measure and satisfying the condition that the boundary $\partial K$ has zero measure.

Actually, here we will obtain this as a side result, being an immediate corollary of our Theorem \ref{th:Requivalence}, see Remark \ref{c:easy}. Moreover, from our approach the same equivalence follows elegantly for arbitrary bounded measurable sets $K$, normalized to have unit measure and satisfying $|\partial K|=0$, i.e. the compactness criterion can be dropped from conditions of Landau.

Also, non-discrete, but locally Lebesgue-measurable sets arise in the context (in problems of plane geometry e.g.), where the natural density is defined by means of volume, not of cardinality. Then a.u.u.d of a Lebesgue-measurable set $A\subset \RR^d$ is defined as
\begin{equation}\label{RUdensity}
\overline{D}_{K}(A) := \limsup_{r\to\infty} \frac{\sup_{x\in\RR^d}
|A\cap (rK+x)|}{|rK|}~.
\end{equation}

That motivates a further extension: we can consider asymptotic uniform upper densities of  \emph{measures}, say some measure $\nu$ 
and not only sequences $S$ or sets $A \subset \RR^d$. So, a general formulation in $\RR^d$ (or $\ZZ^d$) would thus be (writing $|\cdot|$ for the Lebesgue measure in $\RR^d$ or for the cardinality measure in $\ZZ^d$),
\begin{equation}\label{def:Rnu-density}
\overline{D}_{K}(\nu):= \limsup_{r\to\infty} \frac{\sup_{x\in\RR^d}\nu(rK+x)}{|rK|}~,
\end{equation}
To make sense, it is only needed here that the measure $\nu$ is \emph{locally} a finite measure. For simplicity, in this work we will only consider nonnegative measures, so that mentioning measure will be understood as such, always. (Even the very consideration of measures is above the needs of the present work.) It is, however, not an essential restriction that the measure have to be a Borel measure (i.e. the family of measurable sets contain the Borel sigma-algebra) -- we can consider the outer measure $\overline{\nu}$ arising from $\nu$.

However, if we want to have translation invariance of a.u.u.d. (which is a basic requirement towards any reasonable such density notion), then taking $\sup_{x\in \RR^n}$ inside leaves us with not much choice regarding the measure in the denominator: at least asymptotically we need to have it translation invariant, too. This, in turn, more or less determines the measure, too, if we want it to be a Borel measure with finite values on compact sets. In fact, in any locally compact Abelian group, such a translation invariant measure is unique (and thus is called \emph{the} Haar measure) up to a constant factor -- in particular in $\RR^d$ it must be the Lebesgue measure $\lambda$ and in $\ZZ^d$ it must be the counting measure $\#$.

The notion \eqref{def:Rnu-density} of a.u.u.d. indeed remains translation invariant. E.g. in \eqref{RUNdensity} $\nu:={\rm \#}$ is the cardinality or counting measure of a set $S$, while in \eqref{RUdensity} $\nu:=\lambda|_A$ is the trace of $\lambda$ on the measurable set $A\subset \RR$. In fact the point of view of measures, at least as concerns $\nu:=\sum_{s\in S} \delta_s$ with Dirac measures $\delta_s$ placed at the points of $s\in S$, has already been taken by Kahane himself in \cite[page 303]{Kahane1} under the name "measure caract\'eristique".

\section{The appearance of the notion of a.u.u.d. on LCA groups}\label{sec:appearence}

Starting from 2003, we aimed at extending the notion of a.u.u.d. to locally compact Abelian groups (LCA groups henceforth). Our work directly stemmed out from our interest in extending, to LCA groups, some results on the so-called "Tur\'an extremal problem". We indeed succeeded to extend at least the "packing type estimate" of \cite{kolountzakis:groups} from compact groups and $\RR^d$ and $\ZZ^d$ to general LCA groups, see \cite{LCATur} and \cite{dissertation}. Further, in a recent work \cite{Berdysheva} we have similarly analyzed the Delsarte extremal problem and its relation to packing. Note that the Delsarte extremal problem proved to be the precise tool to prove the breakthrough result by Viazovska \cite{Viaz} regarding the densest ball packing in $\RR^8$, subsequently extended also to $\RR^{24}$ in \cite{CKMRV}.


The notion of a.u.u.d. is a way to grab the idea of a set being relatively considerable, even if not necessarily dense or large in some other more easily accessible sense. In many theorems, in particular in Fourier analysis and in additive problems where difference sets are considered, the a.u.u.d. is the right notion to express that a set becomes relevant in the question. However, previously the notion was only extended to sequences and subsets of the real line, and some immediate relatives like $\NN$, $\ZZ^d$, $\RR^d$, as well as to finite, or at least finitely constructed (e.g. $\sigma$-finite) cases.

A framework where the notion might be needed is the generality of
LCA groups. In recent decades it is more and more realized that
many questions e.g. in additive number theory can be investigated,
even sometimes structurally better understood/described, if we
leave e.g. $\ZZ$, and consider the analogous questions in Abelian
groups. In fact, when some analysis, i.e. topology also has a role
-- like in questions of Fourier analysis e.g. -- then the setting
of LCA groups seems to be the natural framework. And indeed
several notions and questions, where in classical results a.u.u.d.
played a role, have already been defined, even in some extent
discussed in LCA groups. Nevertheless, for long no attempt has
been made to extend the very notion of a.u.u.d. to this setup.


Parallel to the first phase of ours, a research which aimed at extending the very first topic where a.u.u.d. have been used -- concerning conditions for sets being sets of sampling or sets of interpolation -- was successfully conducted in \cite{GrochKutySeip}. The construction there is particularly interesting, because it is also a round-about way of arriving at a general notion of a.u.u.d. -- demonstrating that there was no immediate access, and the construction needed some effort. Indeed, to show that the constructed density is equivalent in $\RR^d$ to the classical one, is not quite obvious and is formulated and proved as \cite[Lemma 8]{GrochKutySeip}.

Actually, the main results of \cite{GrochKutySeip} are formulated under the additional assumption that the dual group $\widehat{G}$ is compactly generated, as it is needed for the construction of a.u.u.d. (The relaxation of this extra condition is then discussed in \cite[Section 8]{GrochKutySeip}. The key is that every bandlimited function $f$ in the studied class of functions lives on a quotient $G/K$ and may be identified with a function $\widetilde{f} \in L^2(G/K)$ with some compact subgroup $K$ such that $G/K$ factors according to the structure theorem of locally compact Abelian groups, generated compactly.) The paper constructs the definition of a.u.u.d. referring to a tricky partial ordering relation of uniformly discrete subsets; it is then mentioned that this definition can equivalently be defined using Haar measures, too. That later equivalent formulation surfacing in \cite[formula (18)]{GrochKutySeip} and the discussion following it hints the one we have worked out along quite a different way; details will be seen below. The possibility of consideration of measures and their a.u.u.d. is mentioned in this discussion, too.

We thank to Professor Joaquim Ortega-Cerd\`a for calling our attention to the (then also quite recent) work \cite{GrochKutySeip} right the day after our initial preprint \cite{density} appeared on the ArXiv. For the overlap thus pointed out that earlier version of our work has never been published in a journal, although later we have published a concise version without proofs in the conference abstract \cite{RAE} and the notion, being instrumental for the mentioned extension of the Tur\'an extremal problem to LCA groups, was also presented in the thesis \cite{dissertation} and the paper \cite{LCATur}.

\section{The first construction of a.u.u.d. in LCA groups}\label{sec:auudLCA}

We will consider two generalizations here. The first applies for
the class of Abelian groups $G$, equipped with a topological
structure which makes $G$ a LCA (locally compact Abelian) group.
Considering such groups are natural for they have an essentially
unique translation invariant Haar measure $\mu_G$ (see e.g.
\cite{rudin:groups}), what we fix to be our $\mu$. By
construction, $\mu$ is a Borel measure, that is, the sigma algebra of
$\mu$-measurable sets contains the sigma algebra of Borel mesurable
sets, denoted by $\BB$ throughout. Furthermore, we will write
$\BB_0$ for the family of open sets with compact closure. Sets $B\in \BB_0$ necessarily have positive, but finite Haar measure. If the topology changes, it is reflected by the corresponding change of the (essentially unique) Haar measure, and so the characteristic property of being finite on $\BB_0$ singles out the respective Haar measure from the family of translation-invariant Borel measures, see \cite[(15.8) Remarks, page 194]{HewittRossI}.

Our heuristics in finding a definition of a.u.u.d. was the following. We wanted to grasp the fact that the set, where we may analyze relative densities of the given set $A$ or measure $\nu$, must grow large (as in case of $\RR$ the dilated copies $rK$ do). However, in general LCA groups neither a standard basic neighborhood of $0$ nor dilations exist. Then we encountered the following nice and basic result in LCA groups, see \cite[2.6.7. Theorem]{rudin:groups} or \cite[(31.36) Lemma]{HewittRossII}.

\begin{theorem}\label{th:Rudinlemma} If $\ve>0$ and\footnote{As is commonly used, in any topological space $X$ we write $Y\Subset X$ if $Y$ is a compact subset of $X$.} $C\Subset G$, then there exists  $V\in\BB_0$ such that\footnote{Note on passing that multiplication is continuous, but assuming only that $V$ is to be Borel measurable may not suffice for our purposes. We are thankful to our referee who called our attention to the fact that then $C+V$ is not necessarily Borel measurable, (not even for compact $C$), as was shown by counterexamples in \cite{ES} and in \cite{Rogers}.} $\mu(V+C)<(1+\ve)\mu(V)$.
\end{theorem}

Thinking of $\RR^d$, it is natural to visualize the content of this lemma as follows. For any given compact set $C$ the difference between $V$ and $V+C$ is just a bounded (compact) perturbation on the boundary of $V$, so if $V$ is chosen quite large, than the change of volume becomes relatively negligible. This suggested us the idea of replacing limits and size restrictions by the trick of division by $\mu(V+C)$, in place of simply $\mu(V)$, in the definition of a.u.u.d., thus leading to \eqref{Cnudensity}. Indeed, if $\mu(V)$, that is $V$, is large enough -- in the sense of the above Theorem \ref{th:Rudinlemma} -- then the increase of $\mu(V)$ to $\mu(V+C)$ does not matter asymptotically; and if $V$ is not enough large, than the division by a larger measure (of $\mu(C+V)$) makes the corresponding quantity out of interest in the search of high relative density (i.e. in the inner supremum). That was our heuristical idea in the construction of the below Definition \ref{def:compactdensity}.

\begin{definition}\label{def:compactdensity}
Let $G$ be a LCA group and $\mu:=\mu_G$ be its Haar measure. If
$\nu$ is another (locally finite\footnote{A measure is called locally finite if the measure of any measurable set, contained in some compact set, is finite.}, nonnegative) measure on $G$ with the sigma algebra of
measurable sets being ${\mathcal S}$, then we define
\begin{equation}\label{Cnudensity}
\overline{D}(\nu) 
:= \inf_{C\Subset G} \sup_{V\in {\mathcal S}
\cap \BB_0} \frac{\nu(V)}{\mu(C+V)}~.
\end{equation}
In particular, if $A\subset G$ is Borel measurable and $\nu=\mu_A$
is the trace of the Haar measure on the set $A$, then we get
\begin{equation}\label{CAdensity}
\overline{D}(A) :=\overline{D}(\mu_A) := 
\inf_{C\Subset G} \sup_{V\in {\mathcal B}_0} \frac{\mu(A\cap V)}{\mu(C+V)}~.
\end{equation}
If $\Lambda\subset G$ is any (e.g. discrete) set and $\gamma
:=\gamma_\Lambda:=\sum_{\lambda\in\Lambda} \delta_{\lambda}$ is
the counting measure of $\Lambda$, then we get
\begin{equation}\label{CLdensity}
\overline{D}^{\#}(\Lambda) := \overline{D}(\gamma_{\Lambda}) := \inf_{C\Subset G} \sup_{V\in \BB_0} \frac{{\rm \# }(\Lambda\cap V)}{\mu(C+V)}~.
\end{equation}
\end{definition}


\begin{remark}\label{c:easy} If we do not want to bother with $\nu$-measurability of sets, i.e. with $V\in \SA$, then we may as well use the outer measure $\overline{\nu}$, defined for arbitrary sets. As for all $C$ we want $C+V$ belong to $\BB$ (for $\mu(C+V)$ to make sense), it is natural to consider $V\in\BB$ only; but if we further drop the condition that $\overline{V}$ be compact, then the definition becomes untractable already in $\RR$. Indeed, then it can easily happen -- and in fact that is what happens normally with a compact $C$ having $\mu(C)>0$ -- that $\mu(V+C)=\infty$; further, it is easy to find cases where also the numerator is infinite. Take e.g. $\nu$ to be the counting measure $\#$ and $\Lambda$ some sequence $\Lambda=\{\lambda_k~:~k\in \NN\}$, say tending to infinity; then it is easy to define a (non-compact, but still measurable) union $V$ of decreasingly small neighborhoods of the points $\lambda_k$ such that the Haar measure of $V$ equals 1, but all of $\Lambda$ stays in $V$, hence the counting measure of $\Lambda\cap V$ is infinite. To avoid such untractable situations, we are thus restricted to $V\in\BB_0$, which simultaneously guarantees $\mu(C+V)>0$, too, hence the fraction in the definition remains meaningful.
\end{remark}

The very first thing one wants to have after such an abstract definition, based only on some tricky heuristics, is to see that it indeed is a generalization of the classical notion.
\begin{theorem}\label{th:Requivalence}
Let $K$ be any bounded, open subset of $\RR^d$ with $|K|=|\overline{K}|=1$ (i.e. assume that $K$ itself is of positive, normalized volume $1$ and its closure is of the same measure). Let $\nu$ be any (nonnegative, locally finite) measure with sigma algebra of measurable sets $\mathcal S$. Then we have
\begin{equation}\label{Rd-equivalance}
\overline{D}(\nu) = \overline{D}_K(\nu)~.
\end{equation}
The same statement applies also to $\ZZ^d$.
\end{theorem}

\begin{remark}\label{c:easy}
In particular, we find that the asymptotic uniform upper density
$\overline{D}_K(\nu)$ does not depend on the choice of $K$, as long as $K$ is bounded and measurable with $|K|=1$ and $|\partial K|=0$.
\end{remark}

\begin{remark}\label{c:bounded} If we further drop the condition that $K$ be bounded, then the statement may fail -- or becomes untractable -- already in dimension 1, i.e. for $\RR$.
\end{remark}

The proof of this equivalence result was stated in \cite{LCATur} as Proposition 1
and fully proved in \cite{density} as Theorem 1 and in \cite{dissertation} as Theorem 3.1. Given that the proof is not available in a journal, for the reader's convenience we give the proof here, too.

\section{Proof of Theorem \ref{th:Requivalence}}\label{sec:proof}

\begin{proof} \noindent \underline{Proof of $\overline{D}(\nu) \geq
\overline{D}_K(\nu)$}.

Assume first that $K$ is a convex body (which is then also bounded, and of boundary measure zero with ${\bf 0}\in \intt K$). Let now $\tau<\tau'< \overline{D}_K(\nu)$ and $C\Subset \RR^d$ be arbitrary. Since $C$ is compact, and ${\bf 0}\in \intt K$, for some sufficiently large $r'>0$ we have $C\Subset r'K$, hence by convexity also $C+rK\subset r'K + r K = (r'+r)K$ for any $r>0$.

Observe that in view of $\tau'< \overline{D}_K(\nu)$ there exist $r_n\to\infty$ and $x_n\in \RR^d$
with $\nu(r_nK+x_n)>\tau'|r_nK|$. With large enough $n$, we also have  $|(r_n+r')K| / |r_nK| = (1+r'/r_n)^d < \tau'/\tau$, hence with $V:=r_n K+x_n$ we find $\nu(V) > \tau'|r_n K| > \tau |(r_n+r')K|= \tau |x_n+r_n K +r'K| \geq \tau |V+C|$. This proves that
$\overline{D}(\nu) \geq \tau$, whence the assertion.

\medskip
The proof is only slightly more complicated for the general case. What we need to observe is that if $B\subset \RR$ is the unit ball, then for $\eta\to 0$ we have $|K+\eta B|\to |\overline{K}|$ (where $K+\eta B=\{x+\eta b~:~x\in K, b\in B\}$ is the usual Minkowski- or complexus sum). So if $C\Subset aB$ holds (with some $a$ chosen sufficiently large) then for any given $\eta>0$ we necessarily have $|rK+ C|\le |rK+aB| = r^d |K+(a/r) B| \le r^d (1+\eta) |\overline{K}| =(1+\eta)|rK|$ for all $r>r_0=r_0(\eta,a)$ (where we have used also that $|\overline{K}|=|K|$).

As above, take $\tau<\tau'< \overline{D}_K(\nu)$ and $\tau'|r_nK| < \nu(r_nK+x_n) $ with $r_n\to \infty$. Next let us apply the above with $\eta:=\tau'/\tau-1$, noting that as $r_n\to \infty$, in particular we have $r_n>r_0$ for $n\ge n_0$. We thus find for all $n\ge n_0$ the inequalities $\tau |r_nK+aB| < \tau (1+\eta) |r_nK| =\tau' |r_nK| <\nu(r_nK+x_n)$, whence also $\tau |r_nK+C| <\nu(r_nK+x_n)$. It follows that $\limsup_{n\to\infty} \nu(r_nK+x_n)/|r_nK+C|>\tau$, for any $\tau< \overline{D}_K(\nu)$, whence even $\overline{D}(\nu) \ge \overline{D}_K(\nu)$, as wanted.

\bigskip

\noindent \underline{Proof of $\overline{D}_K(\nu) \geq
\overline{D}(\nu)$}.

\medskip
We have already shown in \cite[Lemma 2]{LCATur} the following lemma.

\begin{lemma}\label{l:translationlemma} Let $W$ be any Borel measurable
subset of a LCA group $G$ with its closure $\overline{W}$ compact, and let $\nu$ be a \emph{nonnegative}, uniformly locally bounded\footnote{A nonnegative Borel measure $\nu$ is uniformly locally bounded, if for any compact set $K$ there exists a finite constant $c=c_K$ such that $\nu(K+x) \le c$ for all $x\in G$.
} Borel
measure on $G$.

Denote $\overline{D}(\nu)=\rho$. If $\gamma<\rho$
is given arbitrarily, then there exists $x \in G$ such that
\begin{equation}\label{VN+z}
\nu(W+x)\geq \gamma \mu(W).
\end{equation}
\end{lemma}

To prove the inequality $\overline{D}_K(\nu) \geq \overline{D}(\nu)$, and whence Theorem \ref{th:Requivalence}, it remains to choose an arbitrary $\gamma<\overline{D}(\nu)$, put $W:=rK$ with any $r>0$, and apply the Lemma: we will get a translate $W+x=rK+x$ with $\nu(x+rK) \ge \gamma |rK|$. Then it follows that $\inf_{r>0} \sup_{x \in \RR^d} \nu(x+rK)/ |rK| \ge \gamma$, whence also $\overline{D}_K(\nu) \ge \gamma$, and this holding for all $\gamma<\overline{D}(\nu)$ gives the statement.
\end{proof}

\section{An even larger notion of a.u.u.d.}\label{sec:auud-discrete}

Note if we consider the discrete topological structure on any
Abelian group $G$, it makes $G$ a LCA group with Haar measure
$\mu_G={\rm \#}$, the counting measure. Therefore, our notions
above certainly cover all discrete groups. This is the natural
structure for $\ZZ^d$, e.g. On the other hand all $\sigma$-finite
groups admit the same structure as well, unifying considerations.
(Note that $\ZZ^d$ is not a $\sigma$-finite group since it is {\em
torsion-free}, i.e. has no finite subgroups.)

Furthermore, we also introduce a second notion of density as
follows.

\begin{definition}\label{finitedensity}
Let $G$ be a LCA group and $\mu:=\mu_G$ be its Haar measure. If
$\nu$ is another (locally finite, nonnegative) measure on $G$ with the sigma algebra of
measurable sets being ${\mathcal S}$, then we define
\begin{equation}\label{Fnudensity}
\overline{\D}(\nu) := 
\inf_{F\subset G,\,{\rm \#} F<\infty } \sup_{V\in {\mathcal S} \cap \BB_0} \frac{\nu(V)}{\mu(F+V)}~.
\end{equation}
In particular, if $A\subset G$ is Borel measurable and $\nu=\mu_A$
is the trace of the Haar measure on the set $A$, then we get
\begin{equation}\label{FAdensity}
\overline{\D}(A) :=\overline{\Delta}(\mu_A) := 
\inf_{F\subset G,\, {\rm \#} F<\infty } \sup_{V\in \BB_0} \frac{\mu(A\cap
V)}{\mu(F+V)}~.
\end{equation}
If $\Lambda\subset G$ is any (e.g. discrete) set and $\gamma
:=\gamma_\Lambda:=\sum_{\lambda\in\Lambda} \delta_{\lambda}$ is
the counting measure of $\Lambda$, then we get
\begin{equation}\label{FLdensity}
\overline{\D}^{\rm \#}(\Lambda):=\overline{\Delta}(\gamma_{\Lambda}) :=
\inf_{F\subset G,\,{\rm \# } F< \infty} \sup_{V\in \BB_0} \frac{{\rm \# }(\Lambda\cap
V)}{\mu(F+V)}~.
\end{equation}
\end{definition}

The two definitions are rather similar, except that the
requirements for $\overline{\D}$ refer to finite sets only.
Because all finite sets are necessarily compact, \eqref{Cnudensity} of Definition \ref{def:compactdensity} extends the same infimum over a wider family of sets than \eqref{Fnudensity}
of Definition \ref{finitedensity}; therefore we get

\begin{proposition}\label{prop:densitycompari}
Let $G$ be any LCA group, with normalized Haar measure $\mu$. Then we have
\begin{equation}\label{Fd-equivalance-gen}
\overline{\D}(\nu) \ge \overline{D}(\nu)~.
\end{equation}
Furthermore, in a discrete Abelian group $G$ we always have
$\overline{\D}(\nu) = \overline{D}(\nu)~.$
\end{proposition}

The second part is even more obvious, because in discrete groups the Haar measure is the counting measure and the compact sets are exactly the finite sets. So there is no difference for $\ZZ$, e.g. In general, however, the two densities, defined above, may be different.

As for the heuristical idea of grasping growth of $V$ through the above trick of taking $\mu(V+C)$ instead of dilations -- which in general do not exist -- we must admit that in Definition \ref{finitedensity} the heuristics fail. That is the essence of the following straightforward example, showing that indeed $\overline{\Delta}(\nu) > \overline{D}(\nu)$ for some $\nu$ whenever $G$ is not discrete. This we were guessing and V. Totik showed that this is indeed the case.

\begin{proposition}[Totik, \cite{totik}]\label{prop:DeltalargerthanD} If $G$ is a non-discrete LCA group, then there exists a probability measure $\nu$ such that $\overline{\Delta}(\nu) > \overline{D}(\nu)$.
\end{proposition}
\begin{proof} First we find an open set which has small Haar measure (which is clearly not possible if $G$ was discrete.) Take an open neighborhood $U$ of 0 with compact closure (and thus of finite Haar measure $0<\mu(U)<\infty$), and let $U_0:=U$.

First we will construct other neighborhoods $U_k$ inductively for all $k\in\NN$ and with small Haar measure. As $G$ is non-discrete, with any given $k$ the neighborhood $U_k$ contains some point $0\ne x_k\in U_k$ out of 0 itself. As addition is a continuous function from $G\times G \to G$, $0+x_k=x_k\in U_k$, and $U_k$ is open, there exists a neighborhood $V_k$ of 0, so that $V_k\times (V_k+x_k)$ is mapped inside $U_k$. Also, by assumption that $G$ is Hausdorff, there are neighborhoods $W_k$ and $W'_k$ of 0 and $x_k$, respectively, which are disjoint. Therefore, taking now $U_{k+1}:=V_k \cap W_k \cap U_k \cap (W'_k-x_k) \cap (U_k-x_k)$ (which is still a neighborhood of 0) we find that $U_{k+1}\cap(U_{k+1}+x_k) =\emptyset$ while  $U_{k+1}, U_{k+1}+x_k \subset U_k$. It follows that $U_{k+1}$ is an open neighborhood of 0, within $U_k$, and its Haar measure is at most $\mu(U_k)/2$.

Therefore, arbitrarily small Haar measures can be prescribed: if $\eta>0$, there exists some open neighborhood $W$ of 0 such that $0<\mu(W)<\eta$. It also follows by outer regularity of $\mu$ that $\mu(\{0\})=0$ for the one point compact set $\{0\}$, whence for any finite set $F\subset G$ we have $\mu(F)=0$, too.

Now we take $\nu:=\delta_0$ the Dirac measure at 0. Let us compute first $\overline{D}(\nu)$. Consider $C\Subset G$ with a positive measure; clearly then $\sup_{V\in \BB_0} \nu(V)/\mu(V+C)\leq 1/\mu(C)$. In fact, this would be attained for $V:=\{0\}$, which is of measure 0, and has compact closure. So, if it was $V \in \BZ$, then taking infimum over $C$ we would find that $\overline{D}(\nu)=1$ if the group $G$ is compact but non-discrete (and thus is normalized to have $\mu(G)=1$, the infimum actually attained for $C:=G$) and $\overline{D}(\nu)=0$ when $G$ is non-compact (and thus there exist compact sets of arbitrarily large measure).

Next we compute $\overline{\Delta}(\nu)$. The same heuristical argument (with the same choice of $\{0\}$ for $V$) with a finite set $F\subset G$ in place of $C\Subset G$ gives that $\overline{\Delta}(\nu)=\inf_{F\subset G, \#F<\infty} 1/\mu(F)$, which, however, is $1/0=\infty$ whenever $G$ is non-discrete. Even if we restrict $V \in \BZ$, the same result obtains taking first some $V\in\BB_0$ with $0<\mu(V)<\eta$, and then writing $\nu(V)/\mu(V+F)\geq 1/(\# F \eta)$, hence $\sup_{V\in \BB_0} \nu(V)/\mu(V+F) \geq \sup_{\eta>0} 1/(\# F \eta)= \infty$, for any fixed finite subset $F$. On this, not even the subsequent $\inf_F$ can help. Therefore the inequality $\overline{\Delta}(\nu) > \overline{D}(\nu)$ is proved if $G$ is non-discrete.
\end{proof}

Note that here -- contradicting to our original heuristics of $\mu(C+V) \to \infty$ together with $\mu(V)\to \infty$ whenever the defined value of our density is approximated closely -- the sets which exhibit close-to-optimal density are very small ones. Applications of density are used in different contexts; in general in e.g. number theory a density is understood as some form of asymptotic density, with measures tending to infinity, but in e.g. real analysis local densities, over small neighborhoods, are equally important. Even if we have a certain heuristics telling what we would like to grasp, we should be careful not to be misled by our own imaginations: this density, what we have defined above, may be extremal also in small sets sometimes. We will see other instances, too, when the heuristics -- e.g. that "the larger the density is, the better it is for a plausible statement" -- may fail.


\section{Some additive number theory flavored results for difference sets}
\label{sec:additiveresults}

We have already noted that extremal problems of Tur\'an and Delsarte, as well as conditions for sets being sets of sampling or interpolation, can be investigated in the generality of LCA groups by means of the a.u.u.d. properly extended. Here we collect a few other instances, mainly of  number theoretic flavor, where generalizations have also been tried, and where we will apply our general definition to extend known results of more restrictive cases to LCA groups in general.

Let us denote the usual upper density of $A\subset \NN$ as
$\overline{d}(A):=\lim\sup_{n\to\infty} A(n)/n > 0$ with $A(n):=
{\rm \#} (A\cap [1,n])$. Erd\H os and S\'ark\"ozy (seemingly
unpublished, but quoted in \cite{hegyvari:differences} and in
\cite{ruzsa:difference}) observed the following.

\begin{proposition}[Erd\H os-S\'ark\"ozy]\label{prop:ErdSar}
If the upper density $\overline{d}(A)$ of a sequence $A\subset
\NN$ is positive, then writing the positive elements of the
sequence $D(A):=D_1(A):=A-A$ as $D(A) \cap \NN
=\{(0<)d_1<d_2<\dots\}$ we have $d_{n+1}-d_n=O(1)$.
\end{proposition}
This is analogous, but not contained in the following result of
Hegyv\'ari, obtained for $\sigma$-finite groups. An Abelian group
is called $\sigma$-finite (with respect to $H_n$), if there exists
an increasing sequence of {\em finite} subgroups $H_n$ so that
$G=\cup_{n=1}^\infty H_n$. For such a group Hegyv\'ari defines
asymptotic upper density (with respect to $H_n$) of a subset $A
\subset G$ as
\begin{equation}\label{GHdensity}
\overline{d}_{H_n}(A) := \limsup_{n\to\infty} \frac{{\rm \#}
(A\cap H_n)}{{\rm \#} H_n}~.
\end{equation}
Note that for finite groups this is just ${\rm \#} (A\cap G) /
{\rm \#} G$. Hegyv\'ari proves the following \cite[Proposition
1]{hegyvari:differences}.

\begin{proposition}[Hegyv\'ari]\label{thm:hegyvari}
Let $G$ be a $\sigma$-finite Abelian group with respect to the
increasing, exhausting sequence $H_n$ of finite subgroups and let
$A\subset G$ have positive upper density with respect to $H_n$.
Then there exists a finite subset $B\subset G$ so that $A-A+B=G$.
Moreover, we have ${\rm \#}B\le 1/\overline{d}_{H_n}(A)$.
\end{proposition}

F\"urstenberg calls a subset $S\subset G$ in a topological Abelian
(semi)group a {\em syndetic} set, if there exists a compact set
$K\subset G$ such that for each element $g\in G$ there exists a
$k\in K$ with $gk\in S$; in other words, in topological groups
$\cup_{k\in K} Sk^{-1}=G$.

Then he presents as Proposition 3.19 (a) of \cite{Furst} the following.

\begin{proposition}[F\"urstenberg]\label{prop:Furst} Let $S\subset
\ZZ$ with positive a.u.u.d. Then $S-S$ is syndetic.
\end{proposition}

In the following we use the above extended notions of a.u.u.d. on arbitrary LCA groups, and present various generalized versions of the above results. Furthermore, we obtain sharpened variants of these results making use of both density notions.

\section{The first extension of the propositions of Erd\H os-S\'ark\"ozy, of Hegyv\'ari, and of F\"urstenberg\label{sec:additiveproposgeneralized}}

Recalling the definition $\overline{\D}$ from Definition \ref{finitedensity}, in this section we will prove the following result.

\begin{theorem}\label{th:general-hegyvari} If $G$ is a LCA group with Haar measure $\mu$,
and $A\subset G$ has $\overline{\D}(A)>0$, then there exists a finite subset $B\subset G$ so that $A-A+B=G$. Moreover, we can find $B$ with $\# B \le [1/\overline{\D}(A)]$.
\end{theorem}

Theorem \ref{thm:intorgenth}, stated in the Introduction, will be a corollary of this result.

\begin{remark} We need a translation-invariant (Haar) measure, but not the topology or compactness.
\end{remark}

\begin{proof}[Proof of Theorem \ref{th:general-hegyvari}] Assume that $H\subset G$ satisfies
$(A-A)\cap(H-H)=\{0\}$ and let $L=\{b_1,b_2,\dots,b_k\}$ be any
finite subset of $H$. By construction, we have $(A+b_i)\cap
(A+b_j)=\emptyset$ for all $1\le i < j\le k$. Take now $C:=L$ in
the definition of density \eqref{FAdensity} and take
$0<\tau<\rho:=\overline{\D}(A)$. By Definition \ref{finitedensity}
of the density $\overline{\D}(A)$, there are $x\in G$ and $V
\subset G$ open with compact closure -- or, a $V\in \SA$ with
$0<|V|<\infty$ -- satisfying
\begin{equation}\label{AVx}
|A\cap(V+x)|>\tau|V+L|~.
\end{equation}
In the other direction,
\begin{equation}\label{VLetc}
V+L=\bigcup_{j=1}^k \left(V+x+(b_j-x) \right)\supset
\bigcup_{j=1}^k \left( ((V+x)\cap A)+b_j \right)-x
\end{equation}
and as $A+b_j$ (thus also $((V+x)\cap A) +b_j$) are disjoint, and
the Haar measure is translation invariant, we are led to
\begin{equation}\label{kVxA}
|V+L|\ge k|(V+x)\cap A|~.
\end{equation}
Combining \eqref{AVx} and \eqref{kVxA} we are led to
\begin{equation}\label{tauk}
|V+L| > k\tau |V+L|~,
\end{equation}
hence after cancelation by $|V+L|>0$ we get $k<1/\tau$ and so in
the limit $k\le K:=[1/\rho]$. It follows that $H$ is necessarily
finite and $\# H\le K$.

So let now $B=\{b_1,b_2,\dots,b_k\}$ be any set with the property
$(A-A)\cap(B-B)=\{0\}$ (which implies $\# B \le K$) and maximal in
the sense that for no $b'\in G\setminus B$ can this property be
kept for $B':=B\cup\{b'\}$. In other words, for any $b'\in
G\setminus B$ it holds that $(A-A)\cap (B'-B')\ne\{0\}$.

Clearly, if $A-A=G$ then any one point set $B:=\{b\}$ is such a
maximal set; and if $A-A\ne G$, then a greedy algorithm leads to
one in $\le K$ steps.

Now we can prove $A-A+B=G$. Indeed, if there exists $y\in
G\setminus(A-A+B)$, then $(y-b_j)\notin A-A$ for $j=1,\dots,k$,
hence $B':=B\cup\{y\}$ would be a set satisfying
$(B'-B')\cap(A-A)=\{0\}$, contradicting maximality of $B$.
\end{proof}

\begin{corollary} Let $A\subset \RR^d$ be a (measurable) set with
$\overline{\D}(A)>0$. Then there exist $b_1,\dots,b_k$ with $k
\le K := [1/\overline{\D}(A)]$ so that $\cup_{j=1}^k
(A-A+b_j)=\RR^d$.
\end{corollary}

This is interesting as it shows that the difference set of a set of positive density $\overline{\D}$ is necessarily rather large: just a few translated copies cover the whole space.

\medskip

Observe that we have Proposition \ref{prop:Furst} as an immediate consequence of Theorem \ref{th:general-hegyvari}, because $\ZZ$ is discrete, and thus the two notions $\overline{\D}$ and $\overline{D}$ of a.u.u.d. coincide; moreover, the finite set $B:=\{b_1,\dots,b_k\}$ is a compact set in the discrete topology of $\ZZ$. But in fact we can as well formulate the following extension.

\begin{corollary}\label{cor:genFurst} Let $G$ be a LCA group and $S\subset G$ a set with positive a.u.u. density, i.e. $\overline{D}(S)>0$ (where $\overline{D}(S)= \overline{D}(\mu|_S)$, in line with \eqref{CAdensity} above). Then the difference set $S-S$ is a syndetic set: moreover, the set of translations $K$, for which we have $G=S+K$, can be chosen not only compact, but even to be a finite set with $\# K \leq [1/\overline{D}(S)]$ elements.
\end{corollary}

This corollary is immediate, because $\overline{\D}(S)\geq \overline{D}(S)$ according to Proposition \ref{prop:densitycompari}. Note that we have already stated a less precise form of this (without the estimate on the size of $K$), as Theorem \ref{thm:intorgenth} in the Introduction.

\bigskip
This indeed generalizes the proposition of F\"urstenberg. Also this result contains the result of Hegyv\'ari: for on $\sigma$-finite groups the natural topology is the discrete
topology, whence the natural Haar measure is the counting measure, and so on $\sigma$-finite groups Corollary \ref{cor:genFurst} and Theorem \ref{th:general-hegyvari} coincides. Finally, this also generalizes and sharpens the Proposition of Erd\H os and S\'ark\"ozy. Indeed, on $\ZZ$ or $\NN$ we naturally have $\overline{\D}(A)=\overline{D}(A)\geq \overline{d}(A)$, so if the latter is positive, then so is $\overline{D}(A)$; and then the difference set is syndetic, with finitely many translates belonging to a translation set $K\subset \NN$, say, covering the whole $\ZZ$. Hence $d_{n+1}-d_n-1$ cannot exceed the maximal element of the finite set $K$ of translations.

\section{Still another extension of the Lemma of Fürstenberg}\label{sec:Furst-extension-auud}

The above given generalization is satisfactory for discrete groups in particular, for those groups the $\overline{\Delta}$ notion of density matches the $\OD$ notion, and is hence a generalized version of the density used in $\ZZ$ by Fürstenberg. However, for general LCA groups, a generally smaller density, that is $\overline{D}$, is known to be the right generalization. The above Corollary \ref{cor:genFurst} settled the generalization right for this density notion. Here we pass on to a third density notion, more precisely, the same density notion but applied to the discrete "characteristic measure" or "cardinality measure". Again, for discrete groups it matches the above two notions, as is trivial from the fact that in discrete groups the Haar measure is just the counting measure. However, in general LCA groups, the number of elements is considerably smaller than the Haar measure -- in fact the cardinality measure is infinite whenever the Haar measure is positive. Therefore, out of all the a.u.u.d. notions, $\overline{D}^{\#}(S) := \overline{D}(\gamma_{S})$ becomes the largest, and knowing that this upper density is positive, is in general the weakest possible assumption on a set. Nevertheless, we have the following result even with this bigger notion of a.u.u.d..

\begin{theorem}\label{thm:strongFurst} Let $G$ be a LCA group and $S\subset G$ a set with a positive, (but finite) a.u.u.d., regarding now the counting measure of elements of $S$ in the definition of a.u.u.d., i.e. $\overline{D}^{\#}(S) := \overline{D}(\gamma_{S})$ in line with \eqref{CLdensity}. Then the difference set $S-S$ is a syndetic set.
\end{theorem}

\begin{remark} One would like to say that a density $+\infty$ is
"even the better", so that we could drop the finiteness condition from the formulation of Theorem \ref{thm:strongFurst}. However, in non-discrete groups this is not the
case: such a density can in fact behave quite unexpectedly. Consider e.g. the
set of points $S:=\{1/n~:~n\in \NN\}$ as a subset of $\RR$.
Clearly for any compact $C$ of positive Haar (i.e. Lebesgue)
measure $|C|>0$, and for any $V\in \BB_0$ of finite measure and
compact closure, $|V+C|$ is positive but finite. Therefore, whenever
$0\in {\rm int} V$, we automatically have $\#(S\cap V)=\infty$ and
also $\#(S\cap V)/|C+V|=\infty$, hence
$\overline{D}^{\#}(S)=\infty$; but $S-S\subset [-1,1]$. Thus with a compact $B$ it is not possible for $B+S-S$ to cover $G=\RR$, and whence $S-S$ is not syndetic.
\end{remark}

\begin{problem}
The implicitly occurring set of translations $K$, for which we have $G=(S-S)+K$, seems not too well controlled in size by the proof below. However, there follows \emph{some} bound, see Remark \ref{rem:Ksize}. One may want to find the right bound, perhaps even $\mu (K) \leq [1/\overline{D}(S)]$, for an appropriately chosen compact set of translates $K$. This we cannot do yet.
\end{problem}

\begin{proof}[Proof of Theorem \ref{thm:strongFurst}] Even if the proof may be not the optimal one, we consider it worthwhile to present it in full detail, for the auxiliary steps done seem to be rather general and useful statements. Correspondingly, we break the argument in a series of lemmas.

\begin{lemma}\label{l:packdensity} Let $S\subset G$ and assume
$\rho:=\overline{D}^{\#}(S)=\overline{D} (\gamma|_S) \in (0,\infty)$. Consider any
compact set $H \Subset G$ satisfying the "packing type condition"
$H-H\cap S-S =\{0\}$ with $S$. Then we necessarily have $\mu(H)
\leq 1/\overline{D}^{\#}(S)$.
\end{lemma}
\begin{proof}
Let $0<\tau<\rho$ be arbitrary. By definition of $\overline{D}^{\#}(S)$, (using $H$ in place of $C$) there must exist a set $V\in \BB_0$ so that
$\infty>\#(S\cap V)>\tau \mu(V+H)$, therefore also $\#(S\cap
V)>\tau \mu((S\cap V)+H)$. However, for any two elements $s\ne s' \in
(S\cap V)\subset S$, $(s+H)\cap (s'+H) =\emptyset$, since in case
$g\in (s+H)\cap (s'+H)$ we have $g=s+h=s'+h'$, i.e. $s-s'=h-h'$,
which is impossible for $s\ne s'$ and $(H-H)\cap (S-S)=\{0\}$.
Therefore for each $s\in (S\cap V)$ there is a translate of $H$,
totally disjoint  from all the others: i.e. the union $(S\cap
V)+H= \cup_{s\in(S\cap V)}(s+H)$ is a disjoint union. By the
properties of the Haar measure, we thus have $\mu((V\cap
S)+H)=\sum_{s\in(S\cap V)} \mu(s+H)= \#(V\cap S)\mu(H)$.

Whence we find $\#(S\cap V) \geq \tau \#(S\cap V) \mu(H)$. As $\#(S\cap V)>\tau \mu(V+H)$ was positive, we can cancel with it and infer $\mu(H)\leq 1/\tau$. This holding for all
$\tau<\rho =\overline{D}^{\#}(S)$, we obtained that any compact set $H$, satisfying the packing type condition with $S$, is necessarily bounded in measure by $1/\overline{D}^{\#}(S)$.
\end{proof}

\begin{lemma}\label{l:fattening} Suppose that $S-S\cap H-H=\{0\}$
with $\rho:=\overline{D}^{\#}(S)=\overline{D} (\gamma|_S) \in (0,\infty)$ and $H\Subset
G$ with $0<\mu(H-H)$. Then the set $A:=S+H$, (that is the trace of the Haar measure on $A$) has the asymptotic uniform upper density $\overline{D}(\mu|_A)$ not less than $\rho\cdot \mu(H)$.
\end{lemma}
\begin{proof} Let $C\Subset G$ be arbitrary. We want to estimate from below the ratio $\mu(A\cap V)/\mu(C+V)$ for an appropriately chosen $V\in \BB_0$. Let us fix that we will take for $V$ some set of the form $U+H$ with $U\in \BB_0$. Clearly $A\cap V = (S+H)\cap (U+H) \supset (S\cap U)+H$. Now for any two elements $s\ne t \in S$, thus even more for $s,t \in (S\cap V)$, the sets $s+H$ and $t+H$ are disjoint, this being an easy consequence of the packing property because $s+q=t+r \Leftrightarrow s-t=q-r$, which is impossible for $s-t\ne 0$ by condition. Therefore by the properties of the Haar measure we get $\mu((S\cap U)+H)=\sum_{s\in(S\cap U)} \mu(s+H)= \#(S\cap U) \cdot \mu(H)$. In all, we found $\mu(A\cap V)\geq \#(S\cap U) \cdot \mu(H)$.

It remains to choose $V$, that is, $U$, appropriately. For the
compact set $C+H\Subset G$ and for any given small $\ve>0$, by
definition of $\overline{D} (\#|_S;\mu)=\rho$ there exists some
$U\in \BB_0$ such that $\#(S\cap U) >(\rho-\ve) \mu((C+H)+U)$.
Choosing this particular $U$ and combining the two inequalities we
are led to $\mu(A\cap V)\geq (\rho-\ve) \mu(C+H+U) \mu(Q)$, that
is, for $V:=U+H$ written in $\mu(A\cap V)/\mu(C+V)\geq (\rho-\ve)
\mu(H)$.

As we find such a $V$ for every positive $\ve$, the sup over
$V\in\BB_0$ is at least $\rho\mu(H)$, and because $C\Subset G$
was arbitrary, we infer the assertion.
\end{proof}

\begin{lemma}\label{l:ifclearthengood} Suppose that $S-S\cap
H-H=\{0\}$ with $\rho:=\overline{D}^{\#}(S)=\overline{D} (\gamma|_S) \in (0,\infty)$ and
$H\Subset G$ with $0<\mu(H)$. Then there exists a finite set
$B=\{b_1,\dots,b_k\}\subset G$ of at most $k\leq
[1/(\rho\mu(H))]$ elements so that $B+(H-H)+(S-S)=G$. In
particular, the set $S-S$ is syndetic with the compact set of
translates $B+(H-H)$.
\end{lemma}
\begin{proof} By the above Lemma \ref{l:fattening} we have an estimate on the asymptotic uniform upper density of $A:=S+H$ (i.e. $\mu|_A$). But then we may apply Corollary \ref{cor:genFurst} to see that the difference set $(S+H)-(S+H)$ is a syndetic set with the set of translates $B$ admitting $\#B\leq [1/\overline{D}(\mu|_A)]\leq [1/(\rho\mu(H))]$. Because also the set $H$ is compact, this yields that $S$ is syndetic as well, with set of translations being $B+(H-H)$. \end{proof}

One may think that it is not difficult, for a discrete set $S$ of
finite density with respect to counting measure, to find a compact
neighborhood $R$ of 0, so that $R\cap (S-S)$ be almost empty with
0 being its only element. If so, then by continuity of
subtraction, also for some compact neighborhood $H$ of zero with
$(H-H)\subset R$ (and, being a neighborhood, with $\mu(H)>0$, too)
we would have $(H-H)\cap(S-S)=\{0\}$, the packing type condition,
whence concluding the proof of Theorem \ref{thm:strongFurst}.

Unfortunately this idea turns to be naive. Consider the sequence
$S=\{n+1/n ~:~ n\in \NN\} \cup \NN$ (in $\RR$), which has asymptotic uniform upper density 2 with the cardinality measure, whilst
$S-S$ is accumulating at 0.

Nevertheless, this example is instructive. What we will find, is
that sets of \emph{finite} positive asymptotic uniform upper
density cannot have a too dense difference set: it always splits
into a fixed, bounded number of disjoint subsets so that the
difference set of each subset already leaves out a fixed compact
neighborhood of 0. This will be the substitute for the above naive
approach to finish our proof of Theorem \ref{thm:strongFurst}
through proving also some kind of subadditivity of the asymptotic uniform upper density -- another auxiliary statement interesting for its own right.

\begin{lemma}\label{l:partition} Let $H\in \BZ$ and let $S$ have positive but finite asymptotic uniform upper density with respect to cardinality measure, i.e. $\rho:=\overline{D}^{\#}(S) \in (0,\infty)$.

Then there exists a finite disjoint partition $S=\bigcup_{j=1}^n S_j$
of $S$ such that $(S_j-S_j)\cap (H-H) =\{0\}$.

Moreover, for any given $\ve>0$ choosing an appropriate $H\in \BZ$, depending on $\ve>0$, we can even guarantee
that the number $n$ of subsets in the partition is not more than
$(1+\ve)\rho\mu(H-H)$.
\end{lemma}
\begin{proof} Let us start with choosing a compact set $C\subset G$ such that $\sup_{V\in\BB_0}\#(S\cap V)/\mu(C+V) <\rho+\ve/2$. By definition of $\overline{D}^{\#}(S)$ such $C$ exists for all $\ve>0$.

Let $s\in S$ be arbitrary, and put $Q:=H-H$, and consider $R:=s+Q$ for an arbitrary, but fixed $s\in S$. Let let us try to estimate the number of other elements of $S$ falling in $R$. Clearly $R\in \BB_0$, so we have $\#(S\cap R)/\mu(C+R)\leq \sup_{V\in\BB_0}\#(S\cap V)/\mu(C+V) <\rho+\ve/2$. That is, we already have a bound $k:=\#(S\cap R)\leq (\rho+\ve/2) \mu(C+R)$ with the given $C=C(\ve)$, independently of $R$, i.e. of $H$.

Next we show how to obtain the bound $k:=\#(S\cap R)\leq (\rho+\ve)\mu(Q)$ for some appropriate choice of $H$. This hinges upon a variant of the above mentioned Theorem \ref{th:Rudinlemma}, but now in the form that for $V$ we can even take an open set of the form $W-W$, where $W\in \BZ$ is also an open set of compact closure. That variant can be extracted from the proofs cited above. So, for the given compact set $C\Subset G$ and for any given (small) $\eta>0$ there exists $V=W-W$, $W\in\BZ$, so that $\mu(C+V)<(1+\eta)\mu(V)$.

In all, working with the above chosen compact set $C=C(\ve/2)$ belonging to $\ve/2$ and \emph{with an appropriate choice of $H:=W$}, coming from $V=W-W$ and $W$ belonging to $C$ and $\eta$, we even have $k:=\#(S\cap R)\leq (\rho+\ve/2) \mu(C+R)< (\rho+\ve/2)(1+\eta) \mu(Q)$. Note that here the dependence on $C$ disappears from the end formula, but there is a dependence of $H$ on $\eta$ and $C$ through the choice of $V$.

It remains to construct the partition once we have a compact
neighborhood $H$ of 0 and a finite number $k\in \NN$ such that
$\#(S\cap(s+H-H)\leq k$ for any $s\in S$. More precisely, we will construct a disjoint partition with $n\leq k$ parts, so the above estimate of $k$ will also imply the asserted bound $n\leq (1+\ve)\rho\mu(H-H)$.

This is a standard argument. Consider a graph on the points of $S$ defined by connecting two points $s$ and $t$ exactly when $t\in s+H-H$. By virtue of the symmetry of $Q:=H-H$ this happens exactly when $s \in t+ H-H$, so the above definition defines indeed a graph, not just a directed graph, on the points of $S$. In this graph by condition the degree of any point of $s\in S$ is at most $k-1$, as there are at most $k-1$ further points $t$ of $S$ in $s+H-H$. But it is well-known that such a graph can be partitioned into $k$ subgraphs with no edges within any of the induced subgraphs.\footnote{The proof of this is very easy for finite or countable graphs: just start to put the points, one by one, inductively into $k$ preassigned sets $S_j$ so that each point is put in a set where no neighbor of it stays; since each point has less than $k$ neighbors, this simple greedy algorithm can not be blocked and the points all find a place. For larger graphs the same argument works in each connected, (hence countable) component.} That is, the set of points $S$ split into the disjoint union of some $S_j$ with no two points $s,t\in S_j$ being in the relation $t\in s+H-H$, defining an edge between them.

It is easy to see that with this we have constructed the required partition: the $S_j$ are disjoint, and so are $(S_j-S_j)$ and $H-H\setminus\{0\}$, for any $j=1,\dots,k$, because $s-t=h-h'$ implies $t=s+h'-h\in s+H-H$, excluded by construction. This concludes the proof.
\end{proof}

\begin{lemma}[\bf subadditivity]\label{l:subadditivity}
Let $\nu_0=\sum_{j=1}^n \nu_j$ be a sum of measures, all on the common set algebra $\SA$ of measurable sets. Then we have $\overline{D}(\nu_0) \leq \sum_{j=1}^n \overline{D}(\nu_j)$.

In particular, this holds for one given measure $\nu$ and a disjoint union of sets $A_0=\cup_{j=1}^n A_j$, with $\nu_j:=\nu|_{A_j}$, for $j=0,1,\dots,k$. If $\nu=\mu$, this gives $\overline{D}(\cup_{j=1}^{n}A_j) \leq \sum_{j=1}^n \overline{D}(A_j)$.
\end{lemma}


\begin{proof} Let us write $\rho_j:=\OD(\nu_j)$, $j=0,1,\ldots,n$. A.u.u.d. is clearly monotone in the measures considered, therefore all $\nu_j$ have an a.u.u.d $0\leq \rho_j\leq \rho<\infty$.

Let $\ve>0$ be arbitrary, and take $C_j\Subset G$ so that for all $V\in \BB_0$ in the definition of $\overline{D}(\nu_j)$ we have $\nu_j(V)\leq (\rho_j+\ve)\mu(C_j+V)$ for $j=1,\dots,n$. Such $C_j$ exist in view of the infimum on $C\Subset G$ in the definition \eqref{Cnudensity} of a.u.u.d.

Consider the (still) compact set $C:=C_1+\dots+C_n$. By definition
of a.u.u.d. there is $V\in\BB_0$ such that $\nu_0(V)\geq (\rho_0-\ve)
\mu(C+V)$.
Obviously, $\mu(C_j+V)\leq \mu(C+V)$, so on combining the above we obtain
$$
\rho_0  -\ve \leq \frac{\nu_0(V)}{\mu(C+V)} = \frac{\sum_{j=1}^k
\nu_j(V)}{\mu(C+V)}  \leq \sum_{j=1}^k \frac{\nu_j(V)}{\mu(C_j+V)}
\leq \sum_{j=1}^k (\rho_j+\ve),
$$
that is, $\rho_0-\ve\leq \sum_j (\rho_j+\ve)$. This holding for all $\ve$,
we find $\rho_0\leq \sum_j \rho_j$, as it was to be proved.
\end{proof}

{\it Continuation of the proof of Theorem \ref{thm:strongFurst}}.
We take now an \emph{arbitrary} compact neighborhood $H\Subset G$
of $0$, with of course $\mu(H)>0$. By Lemma \ref{l:partition}
there exists a finite disjoint partition $S=\cup_{j=1}^n S_j$ with
$(S_j-S_j)\cap(H-H)=\{0\}$. By subadditivity of a.u.u.d. (that is,
Lemma \ref{l:subadditivity} above), at least one of these $S_j$
must have positive a.u.u.d. $\rho_j$ (with respect to the counting
measure), namely of density $0< \rho/n \leq \rho_j \leq \rho <
\infty$, with $\rho:=\overline{D}^{\#}(S)$.

Selecting such an $S_j$, we can apply Lemma
\ref{l:ifclearthengood} to infer that already $S_j$ -- hence also
$S\supset S_j$ -- is syndetic.
\end{proof}

\begin{remark}\label{rem:Ksize} In fact, the above proof also provides an estimate on the measure of the compact translation set $B+H-H$ exhibiting the syndetic property of $S-S$, or, more precisely, already of some $S_j-S_j$, if we select $H$ suitably. Namely, $\# B\leq 1/\rho_j\mu(H)$, where $\rho_j\geq \rho/n$ and $n\leq (1+\ve)\rho \mu(H-H)$ yield $\mu(B+H-H)\leq (n/\rho) \mu(H) \leq (1+\ve)\mu(H-H)/\mu(H)$.
\end{remark}


\medskip

\noindent {\sc\small
Alfr\' ed R\' enyi Institute of Mathematics, \\
1053 Budapest, Hungary}\\
E-mail: {\tt revesz@renyi.hu}


\begin{thebibliography}{99}


\bibitem{beurling} {\sc A. Beurling}, Interpolation for an interval in R1,
in: \emph{The Collected Works of Arne Beurling, in: Harmonic Analysis},
vol. {\bf 2}, Birkhauser Boston, Boston, MA, 1989.

\bibitem{beurlingB} {\sc A. Beurling}, Balayage of Fourier-Stieltjes transforms,
in: \emph{The Collected Works of Arne Beurling, in: Harmonic Analysis},
vol. {\bf 2}, Birkhauser Boston, Boston, MA, 1989.

\bibitem{Berdysheva} {\sc E. Berdysheva, Sz. Gy. R\'ev\'esz}, {\it Delsarte’s extremal problem and packing on locally compact Abelian groups}, Ann. Sc. Norm. Super. Pisa Cl. Sci. (5) {\bf Vol. XXIV} (2023), 1007--1052.

%

\bibitem{CKMRV} {\sc Cohn, Henry; Kumar, Abhinav; Miller, Stephen D.; Radchenko, Danylo; Viazovska, Maryna}, The sphere packing problem in dimension 24.
\emph{Ann. of Math.} (2) {\bf 185} (2017), no. 3, 1017--1033


\bibitem{Furst} {\sc H. F\"urstenberg}, {\em Recurrence in Ergodic
Theory and Combinatorial Number Theory}, Princeton University
Press, Princeton, 1981.

\bibitem{ES} {\sc Erdős, P.; Stone, A. H.}, On the sum of two Borel sets
\emph{Proc. Amer. Math. Soc.} {\bf 25} (1970), 304--306.

\bibitem{GH}{\sc Gottschalk, Walter Helbig; Hedlund, Gustav Arnold}, \emph{Topological dynamics}. American Mathematical Society Colloquium Publications {\bf 36}, American Mathematical Society, Providence, R. I., 1955. vii+151 pp..

\bibitem{GrochKutySeip} {\sc K. Gr\"ochenig, G. Kutyniok, K. Seip}, Landau's necessary density conditions for LCA groups, \emph{J. Funct. Anal.} {\bf 255} (2008) 1831--1850.

\bibitem{Groemer} {\sc Groemer, Helmut}; Existenzs\"atze f\"ur Lagerungen im Euklidischen Raum. (German) \emph{Math. Z.} {\bf 81} (1963) 260--278.

\bibitem{hegyvari:differences} {\sc N. Hegyv\'ari}, On iterated difference sets in groups, {\em Periodica Mathematica Hungarica} {\bf 43} \# (1-2), (2001), 105--110.


\bibitem{HewittRossI}
{\sc E. Hewitt and K. A. Ross}, \emph{Abstract harmonic analysis}, {\bf I},
Die Grundlehren der mathematischen Wissenchaften, Band {\bf 115}, Springer Verlag, Berlin, G\"ottingen, Heidelberg, 1963.

\bibitem{HewittRossII}
{\sc E. Hewitt and K. A. Ross},
\emph{Abstract harmonic analysis}, {\bf II},
Die Grundlehren der mathematischen Wissenchaften, Band {\bf 152}, Springer Verlag, Berlin, Heidelberg, New York ... Budapest, 1970.

\bibitem{Kahane1} {\sc J.-P. Kahane}, Sur les fonctions moyenne-p\'eriodiques born\'ees, \emph{Ann. Inst. Fourier} {\bf 7} (1957) 293--314.

\bibitem{KahaneThese} {\sc J.-P. Kahane}, \emph{Sur quelques probl\`emes d'unicit\'e
et de prolongement, relatifs aux fonctions approchables par des
sommes d'exponentielles}, Th\'ese de doctorat, Universit\'e de
Paris, S\'erie A, no. 2633, 1954, 102 pages.

\bibitem{KahaneThAnnIF} {\sc J.-P. Kahane}, Sur quelques probl\`emes
d'unicit\'e et de prolongement, relatifs aux fonctions
approchables par des sommes d'exponentielles, \emph{Ann. Inst.
Fourier} {\bf 5} (1954) 39--130.

%
%

\bibitem{kolountzakis:groups} {\sc M. N. Kolountzakis,
Sz.\ Gy.\ R\' ev\' esz}, Tur\'an's extremal problem for
positive definite functions on groups, \emph{J. London Math.
Soc.}, \textbf{74} (2006), 475--496.

\bibitem{Landau} {\sc H. J. Landau}, Necessary density conditions for sampling and interpolation of certain entire functions, \emph{Acta Math.} {\bf 117} (1967) 37--52.

\bibitem{density}{\sc Sz. Gy. R\'ev\'esz}, On asymptotic uniform upper
density in locally compact Abelian groups, \emph{preprint}, see on
ArXive as {\tt\small arXiv:0904.1567}, (2009), 13 pages.

\bibitem{LCATur} {\sc Sz. Gy. R\'ev\'esz}, Tur\'an's extremal problem on locally compact Abelian groups, \emph{Anal. Math.} {\bf 37} (2011), 15-50.

\bibitem{dissertation} {\sc Sz. Gy. R\'ev\'esz}, \emph{Extremal problems for
positive definite functions and polynomials}, Thesis for the
"Doctor of the Academy" degree, April 2009, see at {\tt
http://www.renyi.hu/~revesz/preprints.html}

\bibitem{RAE} {\sc Sz. Gy. R\'ev\'esz}, On  asymptotic  uniform  upper  density  in  locally  compact  Abelian  groups, \emph{Real Analysis Exchange}, {\bf 37} (2012),  no. 1,  24--31.  (in the Supplement \emph{35th  Summer  Symposium Conference Reports}).
    
\bibitem{Rogers} {\sc Rogers, C. A.},
A linear Borel set whose difference set is not a Borel set.
\emph{Bull. London Math. Soc.} {\bf 2} (1970), 41--42.

\bibitem{rudin:groups} {\sc W. Rudin},
{\em Fourier analysis on groups}, Interscience Tracts in Pure and
Applied Mathematics, No. {\bf 12} Interscience Publishers (a
division of John Wiley and Sons), New York-London 1962 ix+285 pp.

\bibitem{ruzsa:difference}{\sc I. Z. Ruzsa},
On difference-sequences, {\em Acta Arith.} {\bf XXV} (1974), 151--157.




\bibitem{totik} {\sc V. Totik}, On comparision of asymptotic uniform upper densities in non-discrete LCA groups, \emph{e-mail to Sz. R\'ev\'esz}, 25 July 2010; also in the Referee report on the thesis \emph{"Extremal problems for positive definite functions and polynomials"} by Sz. Gy. R\'ev\'esz, Hungarian academy of Sciences, Budapest, 2011.

\bibitem{Viaz} {\sc M. S. Viazovska}, The sphere packing problem in dimension $8$, \emph{Ann. of Math.} {\bf 185} (2017), no. 3, 991--1015.



\end{thebibliography}
\end{document}